\documentclass[a4paper,11pt]{article}
\usepackage{graphicx} 
\usepackage{amsthm, amsmath}
\usepackage{tikz,ifthen}
\usepackage{a4wide}
\usepackage{amsfonts}
\usepackage{url}
\usepackage{enumerate}
\usepackage{orcidlink}
\usepackage{hyperref}
\usepackage{authblk} 
\usepackage{xcolor}

\usetikzlibrary{shapes.geometric, patterns, calc, decorations.pathreplacing}

\newtheorem{theorem}{Theorem}[section]

\newtheorem{conjecture}[theorem]{Conjecture}
\newtheorem{lemma}[theorem]{Lemma}

\newcommand{\range}[1]{\tilde{\rho}(#1)}

\DeclareMathOperator{\rc}{rc}

\newcommand{\COMMENT}[1]{}
\renewcommand{\COMMENT}[1]{\textcolor{red}{\textbf{/*Comment: } #1\textbf{*/}}} 

\title{Patrolling cop vs omniscient robber}

\author[1]{Nina Chiarelli\,\orcidlink{0000-0002-8169-0925}\footnote{{\tt nina.chiarelli@famnit.upr.si}}}
\author[2]{Paul Dorbec\,\orcidlink{0009-0007-1179-6082}\footnote{{\tt paul.dorbec@unicaen.fr}}}
\author[3]{Miloš Stojaković\,\orcidlink{0000-0002-2545-8849}\footnote{{\tt milos.stojakovic@dmi.uns.ac.rs}}}
\author[4,5]{Andrej Taranenko\,\orcidlink{0000-0002-2438-0496}\footnote{{\tt andrej.taranenko@um.si}}}

\affil[1]{\,FAMNIT and IAM, University of Primorska, Slovenia}
\affil[2]{\,Université Caen Normandie, ENSICAEN, CNRS, Normandie Univ, GREYC UMR 6072, F-14000 Caen, France}
\affil[3]{\,Department of Mathematics and Informatics, Faculty of Sciences,\newline
University of Novi Sad, Serbia}
\affil[4]{\,Faculty of Natural Sciences and Mathematics, University of Maribor, Slovenia}
\affil[5]{\,Institute of Mathematics, Physics and Mechanics, Ljubljana, Slovenia}

\date{}

\begin{document}

\maketitle

\begin{abstract}
We study a variant of the classical Cops and Robbers game with one cop and one robber. The cop follows a fixed walk on the graph, called a patrol, that is chosen before the game begins. The robber is omniscient and knows the entire patrol in advance. A capture occurs when the robber comes within a given distance of the cop, and this distance is referred to as the capture distance. 

The patrol capture radius, $\range{G}$, is the minimum radius of capture required for the cop to always be able to capture the robber on a connected graph $G$, under optimal play. We initiate a systematic study of this parameter for several graph classes. We determine the exact value of $\range{G}$ for trees, establish upper and lower bounds for grids, and analyze the parameter for various families of chordal graphs, including interval graphs and caterpillars. Along the way, we develop general tools and structural results that may be of independent interest for the study of pursuit-evasion games with predetermined patrols and limited information.
\end{abstract}

\noindent {\bf Key words:} Cops and Robber game, pursuit-evasion game, radius of capture

\medskip\noindent
{\bf AMS Subj.\ Class:} 
05C57, 
91A24, 
05C05 

\section{Introduction}

\medskip

Introduced independently by Nowakowski and Winkler~\cite{NW-1983} and Quilliot~\cite{Qu-1983}, the now classical cop and robber game has gained massive popularity and has been intensively studied by numerous authors, see, e.g.,~\cite{BKP-2012, G-2010}. Also see the book by Bonato and Nowakowski~\cite{BoNoBook} for a comprehensive review of the topic. The cop and robber game is a pursuit-evasion game played on the vertices of finite undirected graphs. First, the cop chooses his starting position, then the robber occupies another vertex. They alternately either move to an adjacent vertex or stay in the same vertex. Both players are aware of the position of the opposing player throughout the game. The aim of the cop is to capture the robber, i.e., to occupy the vertex that is occupied by the robber, while the robber wants to evade the cop. A graph $G$ is a {\em cop-win} graph if the cop has a strategy to capture the robber after a finite number of moves regardless of the robber's playing strategy.  

Since the introduction of this game, many variants appeared in the literature. The first and most natural variant replaces a single cop with a set of cops---the minimum number of cops needed to capture the robber in a graph $G$, called the {\em cop number} of $G$ and denoted $c(G)$, was introduced by Aigner and Fromme~\cite{AF-1984}. 
A variant in which a cop can catch the robber from a distance, known as the distance-$k$ cop and robber game, was introduced by Bonato and Chiniforooshan~\cite{BC-2009}. This variant changes the condition for the cop's win; in this case he wins if he is at distance at most $k$ from the robber. The \emph{distance $k$ cop number} of a graph $G$, $c_k(G)$, is the minimum number of cops necessary to win for a given distance $k$. This game was further studied by Chalopin et al.~\cite{CCNV-2011} under the name: {\em the cop and robber game with radius of capture $k$}. Denote by ${\cal{CWRC}}(k)$ the set of graphs in which the cop has a winning strategy in the cop and robber game with radius of capture $k$ regardless of how the robber plays. Chalopin et al.~\cite{CCNV-2011} considered the cop and robber game with radius of capture $1$ on bipartite graphs, and characterized the bipartite graphs in ${\cal{CWRC}}(1)$. Graphs in ${\cal{CWRC}}(1)$ were also studied by Dang~\cite{Dang-2011}, where the cop and robber game with radius of capture 1 was considered on square-grids, $k$-chordal graphs, and outerplanar graphs. In~\cite{DrIrTa-2025} the \emph{radius of capture number} of a graph $G$, $\rc(G)$, was defined as the smallest $k$ such that $G \in {\cal{CWRC}}(k)$, i.e., $\rc(G)=\min{\{k \mid G \in {\cal{CWRC}}(k)\}}$ and some results on this invariant were presented.

Clarke et al.~\cite{Cl-et-al-2020} introduce and explore another variant of the cops and robber game, fundamentally departing from the classical model primarily through restrictions on the cops' information regarding the robber's location. Namely, in the classical game, both cops and the robber possess perfect information, always knowing each other's precise positions throughout the game. In contrast, Clarke et al.~\cite{Cl-et-al-2020} delves into \emph{zero-visibility} (introduced and studied earlier under the name vertex-to-vertex search by Tošić~\cite{Tosic-1985}) and \emph{$\ell$-visibility} variants. In the zero-visibility game, the robber is completely invisible to the cops unless a cop lands on the same vertex. Then in the \emph{$\ell$-visibility game}, studied also by Ba\v{s}i\'c et al.~\cite{bddgm} and Jones and Kinnersley~\cite{jones-kinn}, cops can only see the robber if the distance between them is at most a fixed parameter $\ell$, while the robber has perfect information about the cops' positions. This key modification fundamentally fractures the mechanics of the traditional game, inherently splitting the pursuers' objective into two conceptually distinct, resource-competing phases: an initial search phase dedicated to locating the target, followed by a tactical pursuit phase aimed at final interception. This new layer requires cops to maneuver to first \emph{see} the robber (i.e., move within distance $\ell$), and only then can they progress to capture him, with $c_\ell(G)$ denoting the smallest number of cops to catch the robber in this setting. If the goal of the cops is just to see the robber, then $c'_\ell(G)$ denotes the smallest number of cops to accomplish that. 

A crucial and at first sight somewhat counter-intuitive result in this setting is that seeing the robber does not automatically imply his capture, which is a significant departure from the classical game's dynamics. For instance, on a $C_4$ cycle in the $1$-visibility game, one cop suffices to see the robber,  $c'_1(C_4)=1$, but two are required for capture, and $c_1(C_4)=2$. Clarke et al.~proved that this gap persists even on cop-win graphs---where a single cop dominates under perfect information---showing that the information deficit alone can enforce a resource penalty of $c_\ell(G) = c'_\ell(G) + 1$. Furthermore, the constraints of $\ell$-visibility fundamentally alter the role of monotonicity in search strategies. Because pursuers in limited-visibility settings must frequently employ ``vibrating'' tactics to secure structural bottlenecks, previously cleared vertices are inherently subjected to temporary recontamination. To formally encapsulate this dynamic, Clarke et al.~introduced the concept of \emph{weak monotonicity}, denoted by the parameter $mc_\ell(G)$.

Finally, a particularly interesting finding established in~\cite{Cl-et-al-2020} is that the difference between the number of cops required in the $\ell$-visibility game and the number needed in the classical game ($c_\ell(G) - c(G)$) can be arbitrarily large, demonstrating how profoundly the visibility constraint can increase the resources necessary for success. Building upon this structural paradigm, Ba\v{s}i\'c et al.~\cite{bddgm} and Jones and Kinnersley~\cite{jones-kinn} independently answered a core open question posed by Clarke et al., establishing that the penalty between capturing and merely seeing the robber, $c_\ell(G) - c'_\ell(G)$, can also grow arbitrarily large.

Dumitrescu et al.~\cite{dumitrescu-et-al} study a variant of the cops and robber game in which the cops choose their walks in advance and publicly announce them before the game starts. Their work arises from the study of a discretized, offline version of the Lion and Man game (see~\cite{croft}) on graphs. These results were later extended by Brass et al.~\cite{brass-et-al}.

Here, we consider a variant of the game blending limited visibility, distance-based capture, and a predefined, publicly announced walk for the cop. 
Conceptually, a single cop follows a fixed, deterministic walk (a \emph{patrol}), while an \emph{omniscient} robber tries to evade capture forever using his complete knowledge of the cop's patrol sequence. To compensate for his predictability, the cop possesses a radius of capture $\rho \ge 0$. The central question of this paper is determining the minimum radius of capture required for the cop to successfully catch the robber. We refer to this parameter as the \emph{patrol capture radius} of $G$, denoted by $\range{G}$. The formal turn-by-turn rules of the game and the precise definition of this invariant are deferred to Section \ref{sec:preliminaries}.

We note that another way to conceptualize $\range{G}$ is through the parameter $c'_\ell(G)$. Namely, we have $\range{G} = \min\{\ell:\, c'_\ell(G) = 1\}$, which, consequently, formally classifies our problem under the visibility cops and robbers umbrella.

Yet another way of defining the same problem is through a \emph{graph cleaning process}, as noted in~\cite{Cl-et-al-2020}. One can forget the robber and instead declare every vertex initially dirty; at each round the cop cleans the closed $\rho$-neighborhood of the vertex he is at, and then remaining dirty vertices spread to adjacent vertices. This cleaning process eventually cleans the graph if and only if the cop on patrol with a patrol capture radius of $\rho$ wins against the robber.

Observe that in our setting, the cop cannot adapt his strategy to the moves of the robber. Some results with zero-visibility use probabilistic strategies, so that there is a non-zero probability that the robber gets caught. In our setting, we force the strategy to be deterministic. Clearly, for the cop to have any possibility of winning, he must visit or at least see every vertex of the graph; otherwise, the robber wins.

Another concept that shares similarities with our cop's patrol is the \emph{watchman's walk} of a graph $G$, which is a minimum closed dominating walk of a single watchman traversing $G$---in other words, any minimum closed walk such that every vertex is either visited or seen from a visited neighbor. Hence, in our terminology, it is the minimum closed cop's patrol that catches  a stationary robber with a radius of capture $1$. The watchman’s walk was introduced by Hartnell et al.~\cite{HaRa-98} in 1998, and~\cite{Dyer2022} studies it on graph products, including grid graphs (Cartesian products of paths).

We proceed as follows. In Section~\ref{sec:trees}, we determine the exact value of $\range{G}$ for trees, as well as for some unicyclic graphs. In Section~\ref{sec:grids}, we establish upper and lower bounds for $\range{G}$ on grids. Finally, in Section~\ref{sec:chordal}, we analyze this parameter for chordal graphs, covering in particular the caterpillars and interval graphs. Along the way, we develop general tools and structural results that may be of independent interest in the study of pursuit-evasion games with predetermined patrols and limited information.

\section{Preliminaries}\label{sec:preliminaries}

We formally define the pursuit-evasion game played on a connected graph $G$ with exactly one cop and one robber. Let $\rho\geq 0$ be an integer that we refer to as radius of capture. The cop has a fixed walk on $G$, his \emph{patrol}, and no information on the robber's position at any time. On the other hand, the robber is \emph{omniscient}, meaning he knows all of the cop's movements in advance (the cop's predetermined patrol plan was handed to the robber), and he tries not to be caught by the cop on the graph using this knowledge.

Initially, the cop starts at the first vertex of the patrol; after that, the robber chooses his starting position. Then, they alternate in moving to a neighboring vertex or staying put, the cop goes first, and he follows his patrol at all times. The cop wins if he can capture the robber, i.e., if he is at distance at most $\rho$ from the robber at any point of the game. The \emph{patrol capture radius} of a graph $G$, denoted by $\range{G}$, is the minimum radius of capture that the cop must possess to ensure that he can capture the robber on $G$ in this setting.

A \emph{weak homomorphism} from a graph $G$ to a graph $H$ is a mapping $\varphi\colon V(G) \rightarrow V(H)$ such that for any two adjacent vertices $u, v\in V(G)$, it holds that $d_H(\varphi(u),\varphi(v)) \leq 1$; i.e., they are mapped either to the same vertex or to adjacent vertices. A subgraph $H$ of a graph $G$ is a \emph{weak retract} of $G$ if there exists a weak homomorphism $\varphi\colon V(G) \rightarrow V(H)$ with $\varphi(u)=u$ for all $u\in V(H)$. The map $\varphi$ is called a \emph{weak retraction}.

We prove the following property that we use extensively in the rest of the paper.

\begin{theorem}\label{th:retract}
    If $H$ is a weak retract of a connected graph $G$ and $\range{H} > \rho$, 
    then $\range{G} > \rho$.
\end{theorem}

\begin{proof}
    Let $H$ be a weak retract of a connected graph $G$, and $\varphi\colon V(G) \rightarrow V(H)$ a weak retraction. We employ a standard shadowing argument.    The robber utilizes his winning strategy on $H$, restricting his movements entirely to the vertices of $H$, and imagining for every position $v_c \in V(G)$ that the cop occupies that the cop is actually on $\varphi(v_c) \in V(H)$.

    Since $H$ is a subgraph of $G$, any valid move for the robber in $H$ is also a valid move in $G$. The cop's walk projected onto $H$ remains valid since, for every edge the cop traverses in $G$, the weak homomorphism dictates a corresponding valid move in $H$ (either remaining stationary or traversing an edge).

    By following this strategy, the robber ensures that the distance between himself and the simulated cop in $H$ is always strictly greater than $\rho$. Crucially, because a weak homomorphism maps every walk in $G$ to a walk in $H$ of no greater length, it never increases shortest-path distances. Since $\varphi(v_r) = v_r$ for any position $v_r$ of the robber, it follows that
    \[d_G(v_r,v_c) \ge d_H(\varphi(v_r),\varphi(v_c)) = d_H(v_r,\varphi(v_c)) > \rho.\]
    Thus, the actual cop never comes within distance $\rho$ of the robber in $G$, preventing capture and proving that $\range{G} > \rho$.
\end{proof}

\section{Trees and unicyclic graphs} \label{sec:trees}

In this section, we investigate the game on trees and related graph families, providing exact characterizations of the patrol capture radius. We begin by defining a local guarding maneuver that allows the cop to systematically clear small subtrees. We then introduce a specific branching structure called a \emph{triod}, which acts as the fundamental evasive subgraph for the robber. By determining the patrol capture radius necessary to defeat the robber on triods, we establish a complete formula for $\range{T}$ on any general tree $T$. Finally, we extend our analysis beyond trees to graphs featuring a central clique or cycle---specifically clique-triods and cycle-triods---demonstrating how these central structures alter the evasion dynamics and the required capture radius. 

We say that the cop \emph{checks} a vertex $v$ if he comes within distance at most $\rho$ from $v$.

\begin{lemma} \label{l:local-guarding}
    Let $G$ be a graph, and let $v \in V(G)$ be a vertex. If the cop has a radius of capture $\rho$, then any patrol that visits $v$ can be modified to check all vertices at distance at most $2\rho + 1$ from $v$ without permitting the robber to enter $v$ during this time.
\end{lemma}

\begin{proof}
    Let $O$ be an arbitrary fixed patrol that visits $v$. For every vertex $u$ at distance at most $2\rho +1$ from $v$ and every time $v$ appears in $O$ we replace the appearance of $v$ with a walk $W_u = (v, w_1, \ldots, w_{\rho}, w_{\rho+1}, w_{\rho}, \ldots, w_1, v)$, where $P=v w_1 \ldots w_{\rho}$ is the sub-path of a shortest $u-v$ path and $w_{\rho+1}$ is a neighbor of $w_{\rho}$ that is on a shortest $u-w_{\rho}$ path.
    
    Note that, during the movement along $P$, the robber clearly cannot enter $v$, as $v$ remains continuously checked by the cop. When the cop is at $w_{\rho}$, the vertex $u$ is at distance at most $\rho+1$ from $w_{\rho}$. The cop then moves one more step towards $u$ on a shortest $u-w_{\rho}$ path to $w_{\rho+1}$, and in his next turn, he moves back to $w_{\rho}$. By moving to $w_{\rho+1}$, the cop successfully checks $u$. 
       
    This maneuver means that $v$ is unguarded (at distance $\rho+1$ from the cop) for exactly one round. This single unguarded round allows the robber to enter $v$ on his turn, but does not give him enough time to leave it. Because the cop moves first in the subsequent round, his return to $w_{\rho}$ instantly places $v$ back within his capture radius $\rho$, catching the robber before he can make his next move. Knowing this, the robber is forced to avoid $v$, as required.
\end{proof}



    
    

A \emph{triod} is a tree $T$ with three leaves, and a unique vertex $v$ of degree three that we refer to as the \emph{origin}, see Figure~\ref{fig:triod-c-triod}~(i). We define $\ell_3(T)$ to be equal to the minimum distance from a leaf to the origin and call it \emph{the triod size}. 

A \emph{clique-triod} is a graph made from a clique on at least $3$ vertices, referred to as the \emph{clique-origin}, and three vertex-disjoint paths, each attached to a different vertex of the clique-origin, see Figure~\ref{fig:triod-c-triod}~(ii). Given a clique-triod $G$, we define $\ell_3(G)$ as the length of the shortest of the three paths.

\begin{lemma} \label{l:robber-special-graphs}
If the cop has a radius of capture $\rho$, then
\begin{enumerate}[(i)] 
    \item the robber can win on all triods $G$ satisfying $\ell_3(G)\ge 2\rho + 2$; and
    \item the robber can win on all clique-triods $G$ satisfying $\ell_3(G)\ge 2\rho+1$.
\end{enumerate}
\end{lemma}

\begin{proof}
In both parts of this lemma, the robber utilizes his knowledge of the cop's predetermined patrol to consistently stay as close to the cop as possible without being caught. Because the cop moves first in every round, the robber must position himself at the end of his turn so that he remains strictly outside the cop's capture radius $\rho$ \emph{after} the cop's upcoming move. 

Hence, if the cop's next move will be towards the robber, the robber moves to a vertex at distance $\rho + 2$ from the cop; this ensures that after the cop advances, the distance drops to $\rho + 1$, keeping the robber safe. If the cop's next move will be away from the robber, the robber safely moves to a vertex at distance $\rho + 1$, which will naturally increase to $\rho + 2$ after the cop moves. Finally, if the cop's next move will not change the distance between them, the robber does not move, maintaining his current safe distance.

\begin{figure}[htb]
    \centering
\begin{tikzpicture}[>=stealth, thick]

    \tikzstyle{vertex}=[circle, fill=black, inner sep=1.5pt]
    \tikzstyle{house}=[line width=4pt, blue!40] 
    
    \begin{scope}[xshift=0cm]
        \node at (0, 4) {(i): Triod $G$};
        
        \node[vertex, label=below:{$v$}] (v) at (0,0) {};
        
        \draw (v) -- (90:3) node[midway, left] {$P_1$};
        \draw (v) -- (210:3) node[midway, above left] {$P_2$};
        \draw (v) -- (330:3) node[midway, above right] {$P_3$};
        
        \draw[xshift=3pt, decorate, decoration={brace, amplitude=5pt}] (90:3) -- (90:1.8) node[midway, right=6pt] {$\rho+1$};
        \draw[blue!40] (90:3) -- (90:1.8) node[midway, left=3pt] {$h_1$};

        \draw[yshift=-3pt, xshift=1pt, decorate, decoration={brace, amplitude=5pt}] (210:1.8) -- (210:3) node[midway, below right] {$\rho+1$};
        \draw[blue!40] (210:3) -- (210:1.8) node[midway, above left] {$h_2$};

        \draw[yshift=3pt, xshift=1pt, decorate, decoration={brace, amplitude=5pt}] (330:1.8) -- (330:3) node[midway, above right] {$\rho+1$};
        \draw[blue!40] (330:3) -- (330:1.8) node[midway, below left] {$h_3$};
        
        \draw[house] (90:1.8) -- (90:3);
        \draw[house] (210:1.8) -- (210:3);
        \draw[house] (330:1.8) -- (330:3);
    \end{scope}

        \begin{scope}[xshift=8cm]
        \node at (0, 4) {(ii): Clique-triod $G$};
        
        \tikzstyle{ocNode}=[circle, fill=black, inner sep=1.3pt]
        
        \node[ocNode, label=above left:{$o_1$}] (o1) at (90:0.4) {};
        \node[ocNode, label=above right:{}] (o2) at (18:0.4) {};
        \node[ocNode, label=right:{$o_2$}] (o3) at (306:0.4) {};
        \node[ocNode, label=left:{$o_3$}] (o4) at (234:0.4) {};
        \node[ocNode, label=above left:{}] (o5) at (162:0.4) {};
        
        \draw (o1) -- (o2) -- (o3) -- (o4) -- (o5) -- (o1);
        \draw (o1) -- (o3) -- (o5) -- (o2) -- (o4) -- (o1);

        
        \node[label=90:{}] (L1) at (90:3) {};
        \draw (o1) -- (L1) node[midway, left] {$P_1$};
        
        \node[label=234:{}] (L2) at (234:3) {};
        \draw (o4) -- (L2) node[midway, above left] {$P_2$};
        
        \node[label=306:{}] (L3) at (306:3) {};
        \draw (o3) -- (L3) node[midway, above right] {$P_3$};

        \draw[xshift=3pt, decorate, decoration={brace, amplitude=5pt}] (90:3) -- (90:1.8) node[midway, right=6pt] {$\rho+1$};
        \draw[blue!40] (90:3) -- (90:1.8) node[midway, left=3pt] {$h_1$};
        \draw[house] (90:1.8) -- (90:3); 

        \draw[yshift=-1pt, xshift=3pt, decorate, decoration={brace, amplitude=5pt}] (234:1.8) -- (234:3) node[midway, below right] {$\rho+1$};
        \draw[blue!40] (234:3) -- (234:1.8) node[midway, above left] {$h_2$};
        \draw[house] (234:1.8) -- (234:3); 

        \draw[yshift=1pt, xshift=3pt, decorate, decoration={brace, amplitude=5pt}] (306:1.8) -- (306:3) node[midway, above right] {$\rho+1$};
        \draw[blue!40] (306:3) -- (306:1.8) node[midway, below left] {$h_3$};
        \draw[house] (306:1.8) -- (306:3); 
    \end{scope}

\end{tikzpicture}
    \caption{An example of a triod and a clique triod, with houses denoted in purple.}
    \label{fig:triod-c-triod}
\end{figure}

\paragraph{(i)}
    Let the origin of the triod be $v$, and let $P_1, P_2, P_3$ denote the three paths joining $v$ to the leaves; we call them \emph{branches}. The $\rho+1$ vertices furthest from $v$ on each path will be called \emph{houses}~$h_1, h_2, h_3$, respectively. See Figure \ref{fig:triod-c-triod}~(i) for reference.
    
    When choosing his initial position, the robber follows the same principle as while moving: he places himself on a vertex at distance either $\rho + 2$ or $\rho + 1$ from the cop, depending on whether the cop's first move will be towards him or not.
    If there are multiple such vertices, the robber chooses a branch that is neither the cop's current branch nor the branch containing the next house that the cop will visit. If the cop is at the origin, we arbitrarily choose a branch that does not contain the next house the cop visits to act as the cop's current branch. 

    Following this strategy, if the robber must choose between two possible moves, he has to be at the origin, and furthermore, the cop must have moved to a vertex at distance $\rho+1$ from the origin with the intention to continue moving towards the origin. In such a case, if the cop is on the branch $P_i$, for some $i\in\{1,2,3\}$, and the next house the cop is going to visit is on the branch $P_j$, for some $j\in\{1,2,3\}$, then the robber moves to a branch $P_k$, for $k\not \in \{i,j\}$. Note that if $i=j$, then the robber can move to either of the two remaining branches; otherwise, there is only one valid branch left for him to choose. For example, if the cop approaches from branch $P_1$ and the next house he is going to visit is $h_2$, the robber moves toward $h_3$.
    
    Suppose, for a contradiction, that the cop catches the robber with this strategy. This means that prior to the final turn, the cop moves to a vertex at distance $\rho+1$ from the robber, the robber cannot shadow him at distance $\rho+2$, and the cop makes one more move towards the robber, thus decreasing the distance to $\rho$. But this implies that the robber has reached a leaf, and the cop has entered the house in which the robber is hiding. This further implies that the last time the robber had to leave the origin, he chose the wrong direction. Indeed, the cop could not have visited any other house in the meantime, because the origin is at distance at least $\rho+2$ from any other house, yielding a contradiction.

\paragraph{(ii)} 
    Let $o_1, o_2, o_3$ be the three vertices in the clique to which the paths are attached, and let the paths joining them to the leaves be referred to as \emph{branches}, $P_1, P_2, P_3$, respectively. For each $i\in\{1,2,3\}$, the $\rho+1$ vertices furthest from $o_i$ on the path $P_i$ will be called the \emph{house}~$h_i$. See Figure \ref{fig:triod-c-triod}~(ii) for reference. Because $o_1, o_2$, and $o_3$ are pairwise adjacent, the robber can always step directly between them; consequently, his strategy will completely ignore any other vertices in the clique (if they exist).

    Similarly to (i), when choosing his initial position, the robber places himself on a vertex at distance either $\rho + 2$ or $\rho + 1$ from the cop, depending on whether the cop's first move will be towards him or not.
    If there are multiple such vertices, the robber chooses a branch that is neither the cop's current branch nor the branch containing the next house the cop will visit.

    Following this strategy, the robber will only ever occupy vertices on the branches or the attachment points $o_1, o_2, o_3$. Thus, he may have to choose between two possible moves only if he is at $o_i$, for some $i\in\{1,2,3\}$, and one of the following two cases occurs:

    Case 1. The robber is at $o_i$ and the cop is on the branch $P_i$ at distance $\rho+1$, with an intention to move towards the robber on the next move. The robber considers the next house $h_j$ (different from $h_i$) that the cop is going to enter, and then he moves to $o_k$, where $k\not\in\{i,j\}$.

    Case 2. The robber is at $o_i$, and the cop is on $P_j$, $j\not= i$, at distance $\rho+2$ from the robber, with the intention to move towards the robber on his next move. The robber considers the next house $h_x$ (different from $h_j$) that the cop is going to enter, and then he either stays at $o_i$, if $i\not=x$, or he moves to $o_k$, $k\not\in\{i,j\}$, if $i=x$.

    Note that whenever the cop enters a house, the robber positions himself in such a way that, if the cop moves towards the robber, the robber can retreat to the branch that does not contain the house that the cop will enter next. 

    Now again, as the robber maintains a safe distance from the cop at all times, the only way for him to get caught is if he is cornered in a leaf with the cop approaching within a distance of $\rho$. But that means that the cop entered the house the robber is in, which cannot happen.
\end{proof}

For a general graph $G$ and a vertex $v$ of degree at least 3, we define $\ell_3(v)$ to be the maximum $\ell_3(T)$ over all triods $T$ with origin $v$ induced in $G$. If $v$ is of degree less than 3, then $\ell_3(v)=0$.

\begin{theorem}\label{th:trees}
    For a tree $T$, 
    \[\range{T} = \max_{v\in V(T)}\left\lfloor \frac{\ell_3(v)}{2}\right\rfloor.\]
\end{theorem}

\begin{proof}
    Let $\rho=\max_{v\in V(T)}\left\lfloor \frac{\ell_3(v)}{2}\right\rfloor$. 
      
    We first provide a winning strategy for the cop when his radius of capture is $\rho$.
    Let $\mathcal{L}=\max_{v\in V(T)}\ell_3(v)$ and let $S=\{v\in V(T) \mid \ell_3(v)=\mathcal{L}\}$. In words, the set $S$ contains the origins of all triods in $T$ with the maximum triod size. By~\cite[Lemma 3]{GrMoTa25}, all vertices of $S$ lie on the same path in $T$. Take a longest path $P$ in $T$, say $v_0, v_1, \ldots, v_m$, such that it includes all vertices of $S$. 
 
    The cop's fixed patrol is formed as follows. He proceeds sequentially along the path $P$ from $v_0$ to $v_m$. At each vertex $v_i$, before advancing to $v_{i+1}$, the cop systematically clears all subtrees attached to $v_i$ excluding those attached to $v_{i-1}$ and $v_{i+1}$. We claim that there is no subtree attached to $v_i$ of depth $d \ge 2\rho+2$. We evaluate this based on the position of $v_i$ along $P$:
    \begin{itemize}
        \item When $i\le 2\rho+1$, a subtree of depth $d \ge 2\rho+2$ would induce a path from $v_m$ to the subtree's deepest leaf of length $(m-i) + d \ge m - (2\rho+1) + (2\rho+2) = m+1 > m$. This contradicts that $P$ is a longest path.
        \item When $2\rho+1 <  i < m - (2\rho+1)$, we have $i \ge 2\rho+2$ and $m-i \ge 2\rho+2$. If $d \ge 2\rho+2$, then $v_i$ acts as the origin of a triod with three branches of length at least $2\rho+2$. This implies $\ell_3(v_i) \ge 2\rho+2$, which yields $\lfloor \ell_3(v_i)/2 \rfloor \ge \rho+1$, directly contradicting the definition of $\rho$.
        \item When $i \ge m-(2\rho+1)$, a subtree of depth $d \ge 2\rho+2$ would induce a path from $v_0$ to the subtree's deepest leaf of length $i + d \ge m - (2\rho+1) + (2\rho+2) = m+1 > m$, again contradicting that $P$ is a longest path.
    \end{itemize}

    Because every subtree attached to $v_i$ has depth at most $2\rho+1$, the cop can check every vertex within those subtrees by repeatedly applying the maneuver from Lemma~\ref{l:local-guarding} while anchored at $v_i$. This maneuver ensures that $v_i$ is never left unguarded long enough for the robber to safely cross through it. Thus, the subtrees are completely cleared without the robber ever being able to escape to the previously checked regions of the tree. Pushed along the path, the robber is caught, at the latest, when the cop reaches $v_m$. Hence, $\range{T}\leq \rho$.

    To prove that $\range{T}\geq \rho$, we must show that the robber wins against a cop with a capture radius of $\rho-1$. Let $v\in V(T)$ be such that $\ell_3(v) = \max_{u\in V(T)} \ell_3(u) = \mathcal{L}$, and let $T'$ be any induced triod with origin $v$ that maximizes the triod size. By definition, $\ell_3(T') = \mathcal{L}$. Since $\rho = \lfloor \mathcal{L}/2 \rfloor$, it follows that $\ell_3(T') \ge 2\rho$. 
    
    Define the function $p\colon V(T)\rightarrow V(T')$ to map any vertex $u \in V(T)$ to the unique vertex in $V(T')$ which is closest to $u$. It is easy to see that $p$ is a weak homomorphism, making $T'$ a weak retract of $T$. By Lemma \ref{l:robber-special-graphs}~(i), the robber wins on any triod $H$ against a capture radius of $r$ provided that $\ell_3(H) \ge 2r + 2$. Substituting $r = \rho-1$, the condition becomes $\ell_3(H) \ge 2(\rho-1) + 2 = 2\rho$. Since $\ell_3(T') \ge 2\rho$, the robber possesses a winning strategy on $T'$ against radius $\rho-1$. Using Theorem \ref{th:retract}, this guarantees the robber also wins on $T$, meaning $\range{T} > \rho-1$. Thus, $\range{T} \ge \rho$.
\end{proof}

\paragraph{Examples.} The following two graphs illustrate how adding an edge to a graph $G$ may impact $\range{G}$ in either direction. Figure~\ref{fig:edge-increase} shows a graph $G$ (in black) with the addition of a (red) edge $e$. Before adding the edge, $G$ is a triod with branches of lengths $5, 6,$ and $6$, meaning $\ell_3(G) = 5$. By Theorem~\ref{th:trees}, $\range{G} = \lfloor 5/2 \rfloor = 2$. Adding the edge $e$ transforms the graph into a clique-triod with a $K_3$ origin and three branches of length $5$. By Lemma~\ref{l:robber-special-graphs}~(ii), the robber wins against a cop with radius $2$ because $\ell_3(G+e) = 5 \ge 2(2) + 1$, which increases the patrol capture radius of the new graph to $\range{G+e} = 3$.

\begin{figure}[ht!]
    \centering
    \begin{tikzpicture}[vertex/.style={inner sep=2pt,draw,circle}]
	\begin{scope}
	\node[vertex] (2) at (-5,0){};
	\node[vertex] (3) at (-4,0){};
	\node[vertex] (4) at (-3,0){};
	\node[vertex] (5) at (-2,0){};
	\node[vertex] (6) at (-1,0){};
	\node[vertex] (7) at (0,0){};
	\node[vertex] (8) at (1,0.5){};
	\node[vertex] (9) at (2,0.5){};
        \node[vertex] (10) at (3,0.5){};
        \node[vertex] (11) at (4,0.5){};
        \node[vertex] (12) at (5,0.5){};
        \node[vertex] (13) at (6,0.5){};
        \node[vertex] (14) at (1,-0.5){};
	\node[vertex] (15) at (2,-0.5){};
        \node[vertex] (16) at (3,-0.5){};
        \node[vertex] (17) at (4,-0.5){};
        \node[vertex] (18) at (5,-0.5){};
        \node[vertex] (19) at (6,-0.5){};
	\path[] (2) edge (3);
	\path[] (3) edge (4);
	\path[] (4) edge (5);
	\path[] (5) edge (6);
	\path[] (6) edge (7);
	\path[] (7) edge (8);
	\path[] (8) edge (9);
        \path[] (9) edge (10);
        \path[] (10) edge (11);
        \path[] (11) edge (12);
        \path[] (12) edge (13);
        \path[] (7) edge (14);
	\path[] (14) edge (15);
        \path[] (15) edge (16);
        \path[] (16) edge (17);
        \path[] (17) edge (18);
        \path[] (18) edge (19);
        \path[color=red, line width=1.3pt] (8) edge (14);
        \pgftext[x=1.2cm,y=0cm]{$e$}
	\end{scope}
	\end{tikzpicture}
\caption{Adding edge $e$ increases the required capture radius from $2$ to $3$.}
    \label{fig:edge-increase}
\end{figure}

On the other hand, Figure \ref{fig:edge-decrease} shows a triod $G$ (in black) with a central origin and three branches of length $4$. Thus, $\ell_3(G) = 4$, which gives $\range{G} = \lfloor 4/2 \rfloor = 2$. The addition of edge $e$ (in red) connects a vertex at distance 2 from the origin on one branch to a vertex at distance 1 from the origin on another branch. This shortcut destroys the original triod structure for any meaningful length: to form an induced triod of length 2 or more from the original center, the branches must include those two vertices, but the new edge $e$ acts as a forbidden chord between them. Thus, the maximum induced triod from the original center drops to $\ell_3 = 1$. To find $\range{G+e}$, we must identify the longest valid induced triod remaining anywhere in the graph. By shifting the origin to the new junction at the bottom branch, we can form an induced triod by routing the branches as follows: through the old center to the left leaf (length 5), across the new edge $e$ to the top leaf (length 3), and directly down to the bottom leaf (length 3). By skipping the intermediate vertex on the old top branch, we ensure there are no chords. This new configuration yields a maximum induced triod size of $\ell_3(G+e) = \min(5, 3, 3) = 3$. Consequently, the graph's evasion potential collapses, reducing the patrol capture radius to $\range{G+e} = \lfloor 3/2 \rfloor = 1$.

\begin{figure}[ht!]
    \centering
    \begin{tikzpicture}[vertex/.style={inner sep=2pt,draw,circle}]
	\begin{scope}
	\node[vertex] (3) at (-4,0){};
	\node[vertex] (4) at (-3,0){};
	\node[vertex] (5) at (-2,0){};
	\node[vertex] (6) at (-1,0){};
	\node[vertex] (7) at (0,0){};
	\node[vertex] (8) at (1,0.5){};
	\node[vertex] (9) at (2,0.5){};
        \node[vertex] (10) at (3,0.5){};
        \node[vertex] (11) at (4,0.5){};
        \node[vertex] (14) at (1,-0.5){};
	\node[vertex] (15) at (2,-0.5){};
        \node[vertex] (16) at (3,-0.5){};
        \node[vertex] (17) at (4,-0.5){};
	\path[] (3) edge (4);
	\path[] (4) edge (5);
	\path[] (5) edge (6);
	\path[] (6) edge (7);
	\path[] (7) edge (8);
	\path[] (8) edge (9);
        \path[] (9) edge (10);
        \path[] (10) edge (11);
        \path[] (7) edge (14);
	\path[] (14) edge (15);
        \path[] (15) edge (16);
        \path[] (16) edge (17);
        \path[color=red, line width=1.3pt] (9) edge (14);
        \pgftext[x=1.2cm,y=0cm]{$e$}
	\end{scope}
	\end{tikzpicture}
\caption{Adding edge $e$ decreases the required capture radius from $2$ to $1$.}
    \label{fig:edge-decrease}
\end{figure}

Given positive integers $k$ and $\ell$, the graph $T_{k,\ell}$ is obtained by starting with a cycle $C$ of length $3k$, and locating three vertices $v_1$, $v_2$, $v_3$ on the cycle at distance $k$ from each other. Then three vertex-disjoint paths $P_1$, $P_2$, $P_3$ of length $\ell$ are added, and for each $i\in\{1,2,3 \}$ one endpoint of $P_i$ is identified with $v_i$. See Figure \ref{fig:graph-T49} for an example, where $T_{4,9}$ is depicted.

\begin{figure}[ht!]
    \centering
    \pgfdeclarelayer{vozlisca}
    \pgfdeclarelayer{povezave}
    \pgfdeclarelayer{hise}
    \pgfsetlayers{hise,povezave,vozlisca}
    \begin{tikzpicture}[vertex/.style={inner sep=2pt,fill=white,draw,circle}]
        \pgfmathsetmacro{\n}{12}
        \pgfmathsetmacro{\nm}{\n-1}
        \pgfmathsetmacro{\l}{9}

    \begin{pgfonlayer}{vozlisca}
	   \foreach \i in {0,...,\nm}
            \node [vertex]  (v\i) at ({cos(360/\n * \i)}, {sin(360/\n * \i)}) {}; 

        \foreach \i in {0,...,\l} {
            \node [vertex]  (p1-\i) at ({cos(360/\n * 0)*0.5+cos(360/\n * 0)*(\i+1)*0.5}, {sin(360/\n * 0)*(\i+1)*0.5}) {};
            \node [vertex]  (p2-\i) at ({cos(360/\n * 4)*0.5+cos(360/\n * 4)*(\i+1)*0.5}, {sin(360/\n * 4)*0.5+sin(360/\n * 4)*(\i+1)*0.5}) {};
            \node [vertex]  (p3-\i) at ({cos(360/\n * 8)*0.5+cos(360/\n * 8)*(\i+1)*0.5}, {sin(360/\n * 8)*0.5+sin(360/\n * 8)*(\i+1)*0.5}) {};
        }

        \node[vertex, fill=blue!50!white, label=left:{$v_1$}] at ({cos(360/\n * 0)}, {sin(360/\n * 0)}) {};
        \node[vertex, fill=blue!50!white, label={[label distance=0]280:{$v_2$}}] at ({cos(360/\n * 4)}, {sin(360/\n * 4)}) {};
        \node[vertex, fill=blue!50!white, label={[label distance=0]70:{$v_3$}}] at ({cos(360/\n * 8)}, {sin(360/\n * 8)}) {};

        \node[vertex, fill=orange!50!white, label=above:{$v'_1$}] at ({cos(360/\n * 0)*0.5+cos(360/\n * 0)*(2)*0.5}, {sin(360/\n * 0)*(2)*0.5}) {};
        \node[vertex, fill=orange!50!white, label=left:{$v'_2$}] at ({cos(360/\n * 4)*0.5+cos(360/\n * 4)*(2)*0.5}, {sin(360/\n * 4)*0.5+sin(360/\n * 4)*(2)*0.5}) {};
        \node[vertex, fill=orange!50!white, label=left:{$v'_3$}] at ({cos(360/\n * 8)*0.5+cos(360/\n * 8)*(2)*0.5}, {sin(360/\n * 8)*0.5+sin(360/\n * 8)*(2)*0.5}) {};
    \end{pgfonlayer}

    \begin{pgfonlayer}{povezave}
        \draw (0,0) circle (1);
        \draw (p1-0)--(p1-\l);
        \draw (p2-0)--(p2-\l);
        \draw (p3-0)--(p3-\l);
    \end{pgfonlayer}

    \begin{pgfonlayer}{hise}
        \draw[rounded corners, fill=yellow!50!white] (2.75, -0.25) rectangle (5.75,0.25);
        \draw[rounded corners, fill=yellow!50!white, rotate around={120:(0,0)}] (2.75, -0.25) rectangle (5.75,0.25);
        \draw[rounded corners, fill=yellow!50!white, rotate around={240:(0,0)}] (2.75, -0.25) rectangle (5.75,0.25);
        \node[label=above:{$h_1$}] at (4.25, 0.1) {};
        \node[label=above:{$h_2$}] at (-1.5, 3.25) {};
        \node[label=below:{$h_3$}] at (-1.5, -3.25) {};
    \end{pgfonlayer}
	\end{tikzpicture}
\caption{Example of $T_{4,9}$}
    \label{fig:graph-T49}
\end{figure}

\begin{theorem}\label{thm:cycle-triod} For every positive integers $k$ and $\ell$ it holds that
$$\range{T_{k,\ell}} = \max \left\{ \left\lceil \frac{\ell}{2} + \frac{k-2}{4}\right\rceil, \left\lceil \frac{3k-2}{2}\right\rceil \right\}.$$
\end{theorem}

\begin{proof}
Let $C$ be a cycle of length $3k$, and let $\rho$ denote the cop's radius of capture. As before, the $\rho+1$ vertices furthest from $C$ on paths $P_1,P_2$, and $P_3$ will be called \emph{houses}~$h_1, h_2, h_3$, respectively. Furthermore, let $v'_i$ denote the vertex on $P_i$ at distance $\rho$ from either of the other cycle vertices $v_j$, $j\in\{1,2,3\}\setminus \{i\}$. Figure~\ref{fig:graph-T49} shows an example of $T_{4,9}$ with $\rho=5$, illustrating the labeled houses and the vertices $v'_1, v'_2,$ and $v'_3$.

For the lower bound, first note that, if $\rho < \left\lceil \frac{3k-2}{2}\right\rceil$, the robber can survive indefinitely by simply staying on $C$. Now suppose that $\rho < \left\lceil \frac{\ell}{2} + \frac{k-2}{4}\right\rceil$. To devise a strategy for the robber, we study the sequence of the cop's visits to the houses. For his initial position, the robber chooses the path $P_i$ containing the house that will be visited last by the cop, and stays as close as possible to $C$ while remaining safe. 

During the game, if the robber is on the path $P_i$, he remains there until just before the cop visits the house $h_j$ ($j\not=i$), such that the house the cop will visit after $h_j$ is $h_i$. Then, the last time the cop steps off the vertex $v'_j$ before visiting the house $h_j$, the robber steps onto $C$ and goes straight towards the third path $P_k$, where $k\notin\{i, j\}$. The minimum number of steps the cop needs to enter $h_j$ and reach $v'_j$ again is large enough for the robber to survive this migration, and he can continue switching paths like that indefinitely. Note that, as before, the robber is safe to remain on a path as long as the cop does not enter the respective house.

For the upper bound, let $\rho = \max \left\{ \left\lceil \frac{\ell}{2} + \frac{k-2}{4}\right\rceil, \left\lceil \frac{3k-2}{2}\right\rceil \right\}$. The cop can win using the following strategy: he starts from the vertex in $h_1$ that is closest to $C$, then he goes straight towards the vertex in $h_2$ that is closest to $C$, and once he reaches it, he immediately turns back heading towards the vertex in $h_3$ that is closest to $C$.

Let us show that the described patrol is enough for the cop to win. As the cop initially checks all vertices in $P_1$, once he reaches $v_1$ the robber is not on $P_1$. Furthermore, the choice of $\rho$ implies that the robber cannot be on $C$ either. If he is on $P_2$, he will get caught once the cop sweeps $P_2$ and enters $h_2$. Finally, if the robber is on $P_3$, the only possible time to leave it and attempt to walk on $C$ to switch to another path is while the cop is on $P_2$. But the choice of $\rho$ does not leave enough time for the switch. Hence, the robber must remain on $P_3$ at all times, resulting in him getting caught eventually.
\end{proof}

\begin{theorem}
Let $\ell>0$ be an integer, and let $G$ be a clique-triod with three vertex-disjoint paths of length $\ell$, each attached to a different vertex of the clique-origin. Then 
$$\range{G} = \left\lceil \frac{\ell}{2}\right\rceil.$$
\end{theorem}

\begin{proof}
Using Lemma \ref{l:robber-special-graphs}(ii), one immediately obtains that $\range{G} \geq \left\lceil \frac{\ell}{2}\right\rceil$. For the upper bound, namely $\range{G} \leq \left\lceil \frac{\ell}{2}\right\rceil$, one can use the cop's strategy described in the proof of Theorem~\ref{thm:cycle-triod} with the following changes. In the proof of Theorem~\ref{thm:cycle-triod}, set $k=1$ (thus the cycle is of length 3, i.e., a clique of order 3). For $k=1$, the required radius of capture from Theorem~\ref{thm:cycle-triod} evaluates to $\max \left\{ \left\lceil \frac{\ell}{2} - \frac{1}{4} \right\rceil, \left\lceil \frac{1}{2} \right\rceil \right\}$. Because $\ell$ is a positive integer, $\left\lceil \frac{\ell}{2} - \frac{1}{4} \right\rceil = \left\lceil \frac{\ell}{2} \right\rceil \ge 1$. Hence, the maximum simplifies exactly to $\left\lceil \frac{\ell}{2} \right\rceil$, perfectly coinciding with the value in this theorem. Moreover, if the clique-origin of $G$ is a clique of order larger than $3$, following the same line of thought as in the proof of Theorem~\ref{thm:cycle-triod}, one can show that the robber can neither be on the clique-origin nor move through it without being detected by the cop. 
\end{proof}

\section{Grids} \label{sec:grids}

For grids $P_n\Box P_n$, the cop can win with a radius of capture $\frac{n-1}{2}$.
We can embed a tree to prove that when the cop has a radius of capture less than $\frac{3n}{8}$ the robber can escape.

\begin{theorem}
For $n\le m$, we have
$$\min\left\{\frac{n-3}{2},\frac{2m+n-11}{8}\right\} \le \range{P_n\Box P_m} \le \frac{n}{2}.$$
\end{theorem}

\begin{proof}
Obtaining the upper bound is straightforward. 
Denote by $u_1, \ldots, u_n$ the vertices of the $P_n$ factor, and by $v_1,\ldots,v_m$ the vertices of the $P_m$ factor. 
The cop starts from the vertex $(u_{\lceil\frac{n}{2}\rceil},v_1)$ 
and follows the $P_m$ along $(u_{\lceil\frac{n}{2}\rceil},v_2)$ up to $(u_{\lceil\frac{n}{2}\rceil},v_m)$.

When on $(u_{\lceil\frac{n}{2}\rceil},v_i)$, the vertices $(u_j,v_i)_{1\le j \le n}$ are within the cop's radius of capture.
At the end of his patrol, the cop has covered all the vertices and the robber cannot have moved past him.

\smallskip

\newcommand{\lup}{\ell^\uparrow}
\newcommand{\ldown}{\ell^\downarrow}
\newcommand{\lright}{\ell^\rightarrow}
\newcommand{\Pup}{P^\uparrow}
\newcommand{\Pdown}{P^\downarrow}
\newcommand{\Pright}{P^\rightarrow}
\newcommand{\Hup}{h^\uparrow}
\newcommand{\Hdown}{h^\downarrow}
\newcommand{\Hright}{h^\rightarrow}
\newcommand{\fdown}{\varphi^\downarrow}
\newcommand{\fup}{\varphi^\uparrow}
\newcommand{\pui}[1]{p^\uparrow_{#1}}
\newcommand{\pdi}[1]{p^\downarrow_{#1}}
\newcommand{\pri}[1]{p^\rightarrow_{#1}}

Suppose now that $\rho \le \frac{2m+n-11}{8}$ and $2\rho+3\le n$. 
We prove the robber has a strategy to escape using an ad hoc refinement of Theorem~\ref{th:retract}.
In the following, we assume that $n$ is odd. (If $n$ is even, we can apply the same strategy as for $n-1$.) Furthermore, we may assume without loss of generality that $n \le 4\rho+5$. If $n > 4\rho+5$, the robber can simply restrict his evasion strategy to an induced subgrid $P_{4\rho+5} \Box P_m$, which is a retract of $P_n \Box P_m$. 

Let $x = \frac{4\rho + 5 - n}{2}$. By our assumption bounding $n$, we are guaranteed that $x \ge 0$, ensuring that the coordinates $v_{x+1}$ used in the subsequent projections are always well-defined.

We define two projections $\fup$ and $\fdown$. 
Both map the vertices $(u_i, v_j)$ with $j \ge x+1$ according to the following rule:
\[
(u_i, v_j) \mapsto 
\begin{cases}
 (u_\frac{n+1}{2},v_{j-\frac{n+1}{2}+i}), 
 & \text{ if } i + j \ge \frac{n+1}{2} + x + 1\text{ and }  i \le \frac{n+1}{2} \\ 
 (u_{i + j - x - 1},v_{x+1}), 
 & \text{ if } i + j < \frac{n+1}{2} + x + 1\text{ and }   i \le \frac{n+1}{2} \\ 
 (u_\frac{n+1}{2},v_{j + \frac{n+1}{2} - i}), 
 & \text{ if } j - i \ge x + 1 - \frac{n+1}{2} \text{ and }  i > \frac{n+1}{2} \\ 
 (u_{i - j + x + 1},v_{x+1}),
 & \text{ if } j - i < x + 1 - \frac{n+1}{2} \text{ and }  i > \frac{n+1}{2}. \\ 
\end{cases}
\]

For vertices $(u_i, v_j)$ with $j < x+1$:
 \[
 \fup(u_i,v_j) = \begin{cases}
 (u_{1},v_{i+j-1}), & \text{ if } i+j < x+2 \\
 (u_{i+j-x-1},v_{x+1}), & \text{ if } i+j \ge x+2; \\
 \end{cases}
 \]
 \[
 \fdown(u_i,v_j) = \begin{cases}
 (u_{n},v_{n-i+j}), & \text{ if } i-j \ge n-x-1 \\
 (u_{i+x+1-j},v_{x+1}), & \text{ if } i-j < n-x-1. \\
 \end{cases}
 \]

Figures \ref{fig:grid-lower-bound-fdown} and \ref{fig:grid-lower-bound-fup} show examples of $\fdown$ and $\fup$, respectively.

\begin{figure}[ht!]
    \centering
    \pgfdeclarelayer{vozlisca}
    \pgfdeclarelayer{povezave}
    \pgfdeclarelayer{houses}
    \pgfdeclarelayer{projections}
    \pgfsetlayers{houses,projections,povezave,vozlisca,main}
    \begin{tikzpicture}[vertex/.style={inner sep=2pt,fill=white,draw,circle}]
        \pgfmathsetmacro{\factor}{0.5}
        \pgfmathtruncatemacro{\n}{11}
        \pgfmathtruncatemacro{\m}{14}
        \pgfmathtruncatemacro{\r}{floor(min((2*\m+\n-11)/8,(\n-3)/2))} 
        \pgfmathtruncatemacro{\x}{floor((4*\r+5-\n)/2)}
        \pgfmathtruncatemacro{\npe}{(\n+1)/2}

    \begin{pgfonlayer}{vozlisca}
	   \foreach \i in {1,...,\n}
            \foreach \j in {1,...,\m} 
            {
                \pgfmathtruncatemacro{\vsota}{\i+\j}
                \pgfmathtruncatemacro{\razlika}{\j-\i}               
                \pgfmathtruncatemacro{\mejaA}{(\n+1)/2+\x+1}
                \pgfmathtruncatemacro{\mejaB}{-(\n+1)/2+\x+1}
                \ifthenelse{\j>\x}{
                    \ifthenelse{\i<\npe \OR \i=\npe}{
                        \ifthenelse{\(\vsota>\mejaA \OR \vsota=\mejaA\)}{
                            \node [vertex,fill=red!40!white]  (\i-\j) at (\factor*\j,-\factor*\i) {};
                        }{
                            \node [vertex,fill=green!40!white]  (\i-\j) at (\factor*\j,-\factor*\i) {};
                        }
                    }{
                        \ifthenelse{\razlika > \mejaB \OR \razlika=\mejaB}{
                            \node [vertex,fill=blue!40!white]  (\i-\j) at (\factor*\j,-\factor*\i) {}; 
                        }{
                            \node [vertex, fill=orange!40!white]  (\i-\j) at (\factor*\j,-\factor*\i) {};                    
                        }
                    }
                }
                {
                    \pgfmathtruncatemacro{\mejaC}{\n-\x-2}
                    \pgfmathtruncatemacro{\imj}{\i-\j}     
                    \ifthenelse{\imj > \mejaC}{
                        \node [vertex, fill=magenta!40!white]  (\i-\j) at (\factor*\j,-\factor*\i) {}; 
                    }{
                        \node [vertex, fill=cyan!40!white]  (\i-\j) at (\factor*\j,-\factor*\i) {}; 
                    }
                
                }
                
            } 
        \node[label=above:{$\lup$}] at (\factor, -\factor) {};
        \node[label=below:{$\ldown$}] at (\factor, -\factor*\n) {};
        \node[label={[xshift=0.25cm, yshift=-0.1cm]:$o$}] at (\factor*\x+\factor, -\factor*\npe) {};
        \node[label={[xshift=0.25cm, yshift=-0.1cm]:$\lright$}] at (\factor*\x+\factor*\r*2+\factor*3, -\factor*\npe) {};
    \end{pgfonlayer}

    \begin{pgfonlayer}{projections}
        \foreach \i in {1,...,\n}
            \foreach \j in {1,...,\m} 
            {
                \pgfmathtruncatemacro{\vsota}{\i+\j}
                \pgfmathtruncatemacro{\razlika}{\j-\i}
                \pgfmathtruncatemacro{\npe}{(\n+1)/2}
                \pgfmathtruncatemacro{\mejaA}{(\n+1)/2+\x+1}
                \pgfmathtruncatemacro{\mejaB}{-(\n+1)/2+\x+1}
                \ifthenelse{\j>\x}{
                    \ifthenelse{\i<\npe \OR \i=\npe}{
                        \ifthenelse{\(\vsota>\mejaA \OR \vsota=\mejaA\)}{
                            \pgfmathtruncatemacro{\proA}{\j-(\n+1)/2+\i}
                            \draw[color=red,->]  (\i-\j)to[out=225,in=45](\npe-\proA);
                        }{
                            \pgfmathtruncatemacro{\proB}{\j+\i-\x-1}
                            \pgfmathtruncatemacro{\xx}{\x+1}
                            \draw[color=green, ->]  (\i-\j)to[out=225,in=45](\proB-\xx); 
                        }
                    }{
                        \ifthenelse{\razlika > \mejaB \OR \razlika=\mejaB}{
                            \pgfmathtruncatemacro{\proC}{\j+(\n+1)/2-\i}
                            \draw[color=blue,->]  (\i-\j)to[out=135,in=-45](\npe-\proC);
                        }{
                            \pgfmathtruncatemacro{\proD}{-\j+\i+\x+1}
                            \pgfmathtruncatemacro{\xx}{\x+1}
                            \draw[color=orange, ->]  (\i-\j)to[out=135,in=-45](\proD-\xx);
                        }
                    }
                }
                {
                    \pgfmathtruncatemacro{\mejaC}{\n-\x-2}
                    \pgfmathtruncatemacro{\imj}{\i-\j}     
                    \ifthenelse{\imj > \mejaC}{
                        \pgfmathtruncatemacro{\proE}{\n-\i+\j}
                        \draw[color=magenta,->]  (\i-\j)to[out=-45,in=135](\n-\proE); 
                    }{
                        \pgfmathtruncatemacro{\proF}{-\j+\i+\x+1}
                        \pgfmathtruncatemacro{\xx}{\x+1}
                        \draw[color=cyan, ->]  (\i-\j)to[out=-45,in=135](\proF-\xx);
                    }
                
                }
                
            }
    \end{pgfonlayer}

    \begin{pgfonlayer}{povezave}
        \foreach \i in {1,...,\n}
        {
            \draw[color=gray] (\i-1) -- (\i-\m);  
        }

        \foreach \j in {1,...,\m}
        {
            \draw[color=gray] (1-\j) -- (\n-\j);  
        }

        \pgfmathtruncatemacro{\xx}{\x+1}
        \draw[line width=3, color=blue] (1-1)--(1-\xx)--(\npe-\xx)--(\n-\xx)--(\n-1);
        \draw[line width=3, color=blue] (\npe-\xx)--(\npe-\m);
    \end{pgfonlayer}
	\end{tikzpicture}
\caption{Example $\fdown$ for $P_{11}\Box P_{14}$, here $\rho=3$.}
    \label{fig:grid-lower-bound-fdown}
\end{figure}

\begin{figure}[ht!]
    \centering
    \pgfdeclarelayer{vozlisca}
    \pgfdeclarelayer{povezave}
    \pgfdeclarelayer{houses}
    \pgfdeclarelayer{projections}
    \pgfsetlayers{houses,projections,povezave,vozlisca,main}
    \begin{tikzpicture}[vertex/.style={inner sep=2pt,fill=white,draw,circle}]
        \pgfmathsetmacro{\factor}{0.5}
        \pgfmathtruncatemacro{\n}{21}
        \pgfmathtruncatemacro{\npe}{(\n+1)/2}
        \pgfmathtruncatemacro{\m}{30}
        \pgfmathtruncatemacro{\r}{floor(min((2*\m+\n-11)/8,(\n-3)/2))} 
        \pgfmathtruncatemacro{\x}{floor((4*\r+5-\n)/2)}

    \begin{pgfonlayer}{vozlisca}
	   \foreach \i in {1,...,\n}
            \foreach \j in {1,...,\m} 
            {
                \pgfmathtruncatemacro{\vsota}{\i+\j}
                \pgfmathtruncatemacro{\razlika}{\j-\i}
                \pgfmathtruncatemacro{\npe}{(\n+1)/2}
                \pgfmathtruncatemacro{\mejaA}{(\n+1)/2+\x+1}
                \pgfmathtruncatemacro{\mejaB}{-(\n+1)/2+\x+1}
                \ifthenelse{\j>\x}{
                    \ifthenelse{\i<\npe \OR \i=\npe}{
                        \ifthenelse{\(\vsota>\mejaA \OR \vsota=\mejaA\)}{
                            \node [vertex,fill=red!40!white]  (\i-\j) at (\factor*\j,-\factor*\i) {};
                        }{
                            \node [vertex,fill=green!40!white]  (\i-\j) at (\factor*\j,-\factor*\i) {};
                        }
                    }{
                        \ifthenelse{\razlika > \mejaB \OR \razlika=\mejaB}{
                            \node [vertex,fill=blue!40!white]  (\i-\j) at (\factor*\j,-\factor*\i) {}; 
                        }{
                            \node [vertex, fill=orange!40!white]  (\i-\j) at (\factor*\j,-\factor*\i) {};                    
                        }
                    }
                }
                {
                    \pgfmathtruncatemacro{\mejaC}{\x+3}
                    \ifthenelse{\vsota < \mejaC}{
                        \node [vertex, fill=magenta!40!white]  (\i-\j) at (\factor*\j,-\factor*\i) {}; 
                    }{
                        \node [vertex, fill=cyan!40!white]  (\i-\j) at (\factor*\j,-\factor*\i) {}; 
                    }
                
                }
                
            }   
        \node[label=above:{$\lup$}] at (\factor, -\factor) {};
        \node[label=below:{$\ldown$}] at (\factor, -\factor*\n) {};
        \node[label={[xshift=0.25cm, yshift=-0.1cm]:$o$}] at (\factor*\x+\factor, -\factor*\npe) {};
        \node[label={[xshift=0.25cm, yshift=-0.1cm]:$\lright$}] at (\factor*\x+\factor*\r*2+\factor*3, -\factor*\npe) {};
    \end{pgfonlayer}

    \begin{pgfonlayer}{projections}
        \foreach \i in {1,...,\n}
            \foreach \j in {1,...,\m} 
            {
                \pgfmathtruncatemacro{\vsota}{\i+\j}
                \pgfmathtruncatemacro{\razlika}{\j-\i}
                \pgfmathtruncatemacro{\npe}{(\n+1)/2}
                \pgfmathtruncatemacro{\mejaA}{(\n+1)/2+\x+1}
                \pgfmathtruncatemacro{\mejaB}{-(\n+1)/2+\x+1}
                \ifthenelse{\j>\x}{
                    \ifthenelse{\i<\npe \OR \i=\npe}{
                        \ifthenelse{\(\vsota>\mejaA \OR \vsota=\mejaA\)}{
                            \pgfmathtruncatemacro{\proA}{\j-(\n+1)/2+\i}
                            \draw[color=red,->]  (\i-\j)to[out=225,in=45](\npe-\proA);
                        }{
                            \pgfmathtruncatemacro{\proB}{\j+\i-\x-1}
                            \pgfmathtruncatemacro{\xx}{\x+1}
                            \draw[color=green, ->]  (\i-\j)to[out=225,in=45](\proB-\xx); 
                        }
                    }{
                        \ifthenelse{\razlika > \mejaB \OR \razlika=\mejaB}{
                            \pgfmathtruncatemacro{\proC}{\j+(\n+1)/2-\i}
                            \draw[color=blue,->]  (\i-\j)to[out=135,in=-45](\npe-\proC);
                        }{
                            \pgfmathtruncatemacro{\proD}{-\j+\i+\x+1}
                            \pgfmathtruncatemacro{\xx}{\x+1}
                            \draw[color=orange, ->]  (\i-\j)to[out=135,in=-45](\proD-\xx);
                        }
                    }
                }
                {
                    \pgfmathtruncatemacro{\mejaC}{\x+3}
                    \ifthenelse{\vsota < \mejaC}{
                        \pgfmathtruncatemacro{\proE}{\j+\i-1}
                        \draw[color=magenta,->]  (\i-\j)to[out=45,in=225](1-\proE); 
                    }{
                        \pgfmathtruncatemacro{\proF}{\j+\i-\x-1}
                        \pgfmathtruncatemacro{\xx}{\x+1}
                        \draw[color=cyan, ->]  (\i-\j)to[out=45,in=225](\proF-\xx);
                    }
                
                }
                
            }
    \end{pgfonlayer}

    \begin{pgfonlayer}{povezave}
        \foreach \i in {1,...,\n}
        {
            \draw[color=gray] (\i-1) -- (\i-\m);  
        }

        \foreach \j in {1,...,\m}
        {
            \draw[color=gray] (1-\j) -- (\n-\j);  
        }

        \pgfmathtruncatemacro{\xx}{\x+1}
        \draw[line width=3, color=blue] (1-1)--(1-\xx)--(\npe-\xx)--(\n-\xx)--(\n-1);
        \draw[line width=3, color=blue] (\npe-\xx)--(\npe-\m);
    \end{pgfonlayer}
	\end{tikzpicture}
\caption{Example $\fup$ for $P_{21}\Box P_{30}$, here $\rho=8$.}
    \label{fig:grid-lower-bound-fup}
\end{figure}

Denote $\lup = (u_1, v_1)$, $\ldown=(u_n, v_1)$, $\lright = (u_\frac{n+1}{2},v_{x+2\rho+3})$, and $o=(u_\frac{n+1}{2}, v_{x+1})$. 

Let $\Pup$ denote the path from $o$ to $(u_1, v_{x+1})$ then to $\lup$, and denote its vertices 
$o,\pui{1},\ldots,\pui{2\rho+2}$.
Let $\Pdown$ be the path from $o$ to $(u_n, v_{x+1})$ then to $\ldown$; similarly, denote its vertices 
$o,\pdi{1},\ldots,\pdi{2\rho+2}$.
Finally, let $\Pright$ be the path from $o$ to $(u_\frac{n+1}{2}, v_m)$ and denote its vertices 
$o,\pri{1},\ldots,\pri{m-x-1}$. Note that the paths $\Pup$ and $\Pdown$ are of length $2\rho + 2$ and the path $\Pright$ is of length at least $2\rho + 2$.

Let $T$ be the tree induced by these three paths.
We define the houses $\Hup$ and $\Hdown$ as the sets of all vertices in $P_n\Box P_m$ at distance at most $\rho$ from vertices $\lup$ and $\ldown$, respectively. The house $\Hright$ is defined to be the set of all vertices that project with $\fup$ or $\fdown$ to a vertex on $\Pright$ at distance at least $\rho+2$ from $o$. Note that such vertices are projected to the same vertex by both projections.  

We define the switching area as the set of vertices \(\{(u_\frac{n+1}{2},v_{j}) \mid j\le x+1\}\) (thus including $o$).

To ensure the robber's simulated strategy is valid, we must verify that $\fup$ and $\fdown$ act as weak homomorphisms (i.e., mapping adjacent vertices in the grid to identical or adjacent vertices on $T$) and that swapping between them does not create an inconsistency.

First, we observe that within each piecewise case, a change of $1$ in either the $i$ or $j$ coordinate results in a change of at most $1$ in the projected coordinates, preserving adjacency. We explicitly check the boundary seams for continuity, using $\fup$ as an example. 
Consider the vertical seam between $j=x$ and $j=x+1$. For $i \ge 2$ and $i+x \ge x+2$, the vertex $(u_i, v_x)$ maps to $(u_{i-1}, v_{x+1})$. Moving one step right to $(u_i, v_{x+1})$ transitions the mapping to the $j \ge x+1$ domain. Evaluating the piecewise definition for $j \ge x+1$ at $(u_i, v_{x+1})$ simplifies to $(u_i, v_{x+1})$ across all relevant conditions. The distance on $T$ between $(u_{i-1}, v_{x+1})$ and $(u_i, v_{x+1})$ is exactly $1$. 
Similarly, consider the diagonal seam in the upper half of the grid where $i \le \frac{n+1}{2}$. A vertex satisfying $i+j = \frac{n+1}{2} + x$ falls into the second case of the piecewise definition, mapping to $(u_{\frac{n+1}{2}-1}, v_{x+1})$. Moving one step right or down to an adjacent vertex where $i+j = \frac{n+1}{2} + x + 1$ crosses into the first case of the piecewise definition, which maps to $(u_{\frac{n+1}{2}}, v_{x+1})$. The projected distance remains exactly $1$. Symmetric verifications hold for all other boundaries and for $\fdown$, confirming that both mappings act as valid weak homomorphisms across their entire domains.

Next, we verify the switching area. If the cop is at a vertex $C = (u_{\frac{n+1}{2}}, v_j)$ with $j \le x+1$, we evaluate both projections. 
Using $\fup$, because $\frac{n+1}{2}+j \ge x+2$, the cop is mapped to $(u_{\frac{n+1}{2}+j-x-1}, v_{x+1})$ on $\Pup$. The distance from this projected position to the origin $o$ is $\frac{n+1}{2} - (\frac{n+1}{2}+j-x-1) = x+1-j$. 
Using $\fdown$, because $\frac{n+1}{2}-j \le n-x-1$, the cop is mapped to $(u_{\frac{n+1}{2}+x+1-j}, v_{x+1})$ on $\Pdown$. The distance from this projected position to $o$ is $(\frac{n+1}{2}+x+1-j) - \frac{n+1}{2} = x+1-j$. 

Because both projections place the cop at the exact same distance ($x+1-j$) from the junction $o$, the switching area is well-defined. Furthermore, when the cop is in this switching area, the robber is at a distance of at least $\rho+2$ from the cop. By the geometry of $T$, the robber must be located on the branch $\Pright$. Because the robber is on $\Pright$, the total distance between the robber $R$ and the projected cop in the simulated game is $d_T(R, o) + d_T(o, \fup(C))$. Since $d_T(o, \fup(C)) = d_T(o, \fdown(C))$, swapping the projection dynamically relocates the cop's imagined position from $\Pup$ to $\Pdown$ without altering the cop-robber distance on $T$ by even a single edge. Thus, switching projections does not create a gap in the strategy.

Now to get back to the robber's strategy, he imagines the game is played on $T$ and follows his strategy there.
The only difference is that the projection of the cop's position onto the tree the robber uses differs depending on where the cop is and which of the other two houses he will visit next.

If the cop entered house $\Hright$ last and is going toward house $\Hup$ (resp. $\Hdown$), or vice versa, then the robber uses the projection $\fup$ (resp. $\fdown$) to imagine the cop's position.
If the cop left house $\Hup$ and is moving toward house $\Hdown$, the robber uses projection $\fup$ for the cop's position until the cop reaches the switching area, at which time the robber switches to projection $\fdown$ to imagine the cop's position, and vice versa.
Observe that the vertices in the switching area are projected to vertices at the same distance from $o$ by both projections. Note also that when the cop moves to a vertex in the switching area while traveling between $\Hup$ and $\Hdown$, the robber is at distance $\rho+2$; hence he is on the branch $\Pright$, so both projections are equidistant to his position. Switching the projection does not generate an inconsistency in his strategy.
\end{proof}

\section{Chordal graphs} \label{sec:chordal}

In this section, we consider the problem on the family of chordal graphs. Note that they are particularly interesting for the $\ell$-visibility Cops and Robber game, as Clarke et al.~\cite[Theorem 3.2]{Cl-et-al-2020} proved that in chordal graphs, $c_\ell(G) = c'_\ell(G)$, i.e., that it is sufficient to see the robber to be able to capture him.

We first consider two special subclasses of chordal graphs, namely caterpillars and interval graphs, with the aim of gaining insight towards a more general result for chordal graphs. We conclude this section with a general conjecture for chordal graphs that would generalize the first results.

A graph $G$ is a \emph{caterpillar} if it is a tree and all vertices are within distance 1 of a central path called the \emph{backbone}. The specific case where the patrol capture radius is exactly $0$ corresponds to the vertex-to-vertex search game with a single searcher, $vs(G)=1$, introduced by To\v{s}i\'c~\cite{Tosic-1985}. The following characterization was originally proven in that context; for completeness, we include a different proof utilizing our terminology and the structural retract lemmas.

\begin{theorem}[To\v{s}i\'c~\cite{Tosic-1985}]\label{thm:caterpillar}
Let $G$ be a connected graph. Then $\range{G} = 0$ if and only if $G$ is a caterpillar.
\end{theorem}

\begin{proof}
If $G$ is a caterpillar, let $u_1,\ldots,u_k$ be the longest induced path in $G$, its backbone. The cop starts from $u_1$, goes toward $u_k$ while systematically visiting the legs of the caterpillar. The cop is situated on the backbone every second move, thus the robber situated on the path cannot move past the cop toward $u_1$ without being caught.

Suppose now that $G$ is not a caterpillar. If $G$ is a tree, a well-known characterization by Harary and Schwenk~\cite{HS-1971} states that $G$ must contain a subdivided $K_{1,3}$ as a subgraph. Because any connected subgraph of a tree is a retract, $G$ contains a subdivided $K_{1,3}$ as a retract, which implies $\range{G} \ge 1$ by Theorem~\ref{th:trees}.

If $G$ contains a cycle $C$, the robber can evade a cop with a capture radius of $0$ by restricting his movements to $C$. In any round, the robber has three available vertices to occupy: his current position and his two neighbors on $C$. To survive, he must only avoid the cop's current position and the cop's next position. Because these represent at most two distinct vertices, the robber always has at least one safe option. Thus, he survives indefinitely, yielding $\range{G}\ge 1$, which completes the proof.
\end{proof}

An \emph{interval representation} of a graph is a family of real line intervals assigned to vertices in such a way that two distinct vertices are adjacent if and only if the corresponding intervals intersect. A graph is an \emph{interval graph} if it has an interval representation. One may assume that all intervals in any interval representation of an interval graph are closed, have positive length, and no intervals agree on any endpoint. Caterpillars have been characterized as connected triangle-free interval graphs~\cite{Eckhoff93}.

Let $G$ be an interval graph. Given an interval representation of $G$, for all $v\in V(G)$ let $I_v=[b_v, e_v]$ be the interval corresponding to $v$. A vertex $v$ of $G$ is called an \emph{end-vertex} if there exists an interval representation of $G$ such that $b_v = \min\{ b_u \mid u\in V(G)\}$ or $e_v = \max\{ e_u \mid u\in V(G)\}$, the corresponding interval $I_v$ is called an \emph{end-interval}~\cite{Gim88}. A maximal clique of $G$ is an \emph{end-clique} if it contains an end vertex of $G$. It is easy to see that an end-vertex $v$ of $G$ is always simplicial, i.e., $N[v]$ is a clique. 

\begin{theorem}\label{th:interval}
    If $G$ is a connected interval graph, then $\range{G}=\begin{cases}
        0, & \text{if $G$ is a caterpillar} \\
        1, & \text{otherwise.}
    \end{cases}$
\end{theorem}

\begin{proof}
    If $G$ is a caterpillar, then the claim follows from Theorem \ref{thm:caterpillar}.

    Assume that $G$ is not a caterpillar. Choose an arbitrary interval representation of $G$ and for all $v\in V(G)$ let $I_v=[b_v, e_v]$ be the interval corresponding to $v$. Order the vertices of $G$ with respect to the left end-points in their interval representation, say $v_1, v_2, \ldots, v_n$ where $b_{v_1} \leq b_{v_2} \leq \ldots \leq b_{v_n}$. 

    We claim that if the cop with a radius of capture equal to 1 uses a walk through all vertices, starting from $v_1$, then he can catch the robber. Let the cop start in $c_1=v_1$ and say the robber chooses to be in any other vertex of $G$, say $r_1 \not= c_1$ (otherwise he is immediately caught). If $d(c_1, r_1)=1$ then $\range{G}\leq 1$. 
    
    Assume now $d(c_1, r_1)>1$. Therefore $e_{c_1} < b_{r_1}$. Since in our strategy the cop visits all vertices of $G$, at some point in time, say $t+1$, it must hold that $b_{r_{t+1}} \leq b_{c_{t+1}}$, and for all $t' < t+1$ it holds that  $b_{r_{t'}} > b_{c_{t'}}$. 
    
    If $b_{r_t} \leq e_{c_t}$, then $d(c_t, r_t)=1$ and $\range{G}\leq 1$, again we are done. Otherwise, since $b_{c_t} < b_{r_t}$ and  $b_{r_{t+1}} \leq b_{c_{t+1}}$, $c_{t+1} \in N[c_t]$ and $r_{t+1} \in N[r_t]$, it follows that $I_{c_{t+1}} \cap I_{r_{t+1}} \not= \emptyset$. This means that $d(c_{t+1}, r_{t+1}) \le 1$ (they are on the same clique) and therefore, $\range{G}\leq 1$. By Theorem~\ref{thm:caterpillar} and the fact that $G$ is not a caterpillar, we have that $\range{G}>0$ and the assertion follows.
\end{proof}

A graph $G$ is \emph{chordal} if it has no induced cycles of length greater than $3$. 
Equivalently, chordal graphs admit a perfect elimination ordering, i.e., a bijection 
$\sigma\colon V \rightarrow \{1,\ldots, n\}$ such that the neighborhood of each vertex 
$v$ within $\{ u \mid \sigma(u) > \sigma(v)\}$ induces a clique.
Chordal graphs also admit a clique-tree decomposition, 
i.e., a treelike representation in which each node (usually called \emph{bag}) 
is a maximal clique of $G$. 

All interval graphs are chordal graphs that admit a tree decomposition that form a single path. 
Our intuition is that  the strategy we adopted for trees can also be applied 
on chordal graphs, with a similar strategy. 
However, there is no natural link between the distance separating two vertices in 
a chordal graph and their distance in the tree decomposition. 
We were thus not able to prove it, but we suggest the following conjecture that would settle the case of chordal graphs.

\begin{conjecture}
    Let $G$ be a connected chordal graph and $\rho\ge 2$. We have $\range{G} \ge \rho$ if and only if there is a triod  $T$ with $\ell_3(T) \ge 2\rho$ or a clique-triod $C$ with $\ell_3(C) \ge 2\rho -1$ that is a weak retract of $G$.
\end{conjecture}

Recall that since seeing the robber is equivalent to capturing him in the $\ell$-visibility Cops and Robber game \cite{Cl-et-al-2020}, settling this conjecture in the affirmative would solve the problem for a single cop in both games.

\section*{Acknowledgment} 

This work was initiated at the \emph{3rd Workshop on Games on Graphs} on Rogla, Slovenia in June 2025. We thank the organizers and all the participants for the inspiring atmosphere.

Nina Chiarelli acknowledges partial support by the Slovenian Research and Innovation Agency (research program P1-0404, and research projects N1-0370, and J1-4008).

Miloš Stojaković acknowledges partial support by the Science Fund of the Republic of Serbia, \#7462, \emph{Graphs in Space and Time: Graph Embeddings for Machine Learning in Complex Dynamical Systems (TIGRA)}, and partial support by the Ministry of Science, Technological Development and Innovation of the Republic of Serbia (grants 451-03-33/2026-03/200125 \& 451-03-34/2026-03/200125).

Andrej Taranenko acknowledges the financial support from the Slovenian Research and Innovation Agency (research core funding No. P1-0297, projects N1-0285 and N1-0431).

\end{document}